\documentclass[12pt]{amsart}
\usepackage{amssymb,amsmath,bm}
\usepackage{enumitem}
\newcommand{\eps}{\varepsilon}
\newcommand{\e}{{\rm e}}
\newcommand{\dist}{{\rm dist}}
\newcommand{\dt}{{\rm d}}
\newcommand{\Id}{{\rm Id}}
\newcommand{\dom}{{\rm dom}}
\newcommand{\dive}{{\rm div \,}}

\newcommand{\R}{\mathbb{R}}

\newcommand{\N}{\mathbb{N}}

\newcommand{\B}{{\mathcal B}}
\newcommand{\D}{{\mathcal D}}

\newcommand{\A}{{\mathcal A}}
\newcommand{\K}{{\mathcal K}}

\newcommand{\Q}{{\mathcal Q}}

\newcommand{ \la}{\langle}
\newcommand{ \ra}{\rangle}

 \newcommand{\TT}{{\mathbb T}}
\newcommand{\C}{{\mathcal C}}
\newcommand{\U}{{\mathcal U}}

\newcommand{\pt}{{\partial_t}}

\newtheorem{proposition}{Proposition}[section]
\newtheorem{theorem}[proposition]{Theorem}
\newtheorem{corollary}[proposition]{Corollary}
\newtheorem{lemma}[proposition]{Lemma}
\theoremstyle{definition}
\newtheorem{definition}[proposition]{Definition}
\newtheorem{remark}[proposition]{Remark}

\numberwithin{equation}{section}

\newcommand{\au}[1]{\rm{#1}}
\newcommand{\ti}[1]{\it{#1}}
\newcommand{\jou}[1]{\rm{#1}}
\newcommand{\bk}[1]{\it{#1}}
\newcommand{\no}[3]{\textbf{#1} (#3), #2.}
\newcommand{\eds}[3]{#1, #2, #3.}

\begin{document}

\title[Multivalued Attractors and their Approximation]
{Multivalued Attractors and their Approximation: Applications to the Navier-Stokes equations}

\author[M. Coti Zelati, F. Tone]
{Michele Coti Zelati, Florentina Tone}

\date{\today}

\address{Indiana University - Department of Mathematics
\newline\indent
Rawles Hall, Bloomington, IN 47405, USA}
\email{micotize@indiana.edu {\rm (M.\ Coti Zelati)}}

\address{University of West Florida - Department of Mathematics and Statistics
\newline\indent
Pensacola, FL 32514, USA}
\email{ftone@uwf.edu {\rm (F.\ Tone)}}

\subjclass[2000]{37L05, 37L15, 65P99, 35Q35}

\keywords{Multivalued dynamical systems, global attractors, time discretization, Navier-Stokes equations, implicit Euler scheme}

\maketitle

\begin{abstract}
This article is devoted to the study of  multivalued semigroups and their asymptotic behavior,
with particular attention to iterations of set-valued mappings. After developing 
a general abstract framework, we present an application to a time discretization 
of the two-dimensional Navier-Stokes equations.
More precisely, we prove that the fully implicit Euler scheme generates a  family 
of discrete multivalued  dynamical systems, whose global attractors converge to 
the global attractor of the continuous system as the time-step parameter approaches zero.
\end{abstract}
\section{Introduction}

\noindent The variety of questions related to evolution equations arising from fluid mechanics problems
constitutes  a challenging and fascinating area of mathematics, which has attracted the attention
of a wide number of researchers for many years. One important aspect, among others, is the
understanding of the behavior of solutions to differential equations as time goes to infinity. For
autonomous systems, this translates into the study of the properties of a semigroup of
operators $\{S(t)\}_{t\geq 0}$, also called a dynamical system, acting on a phase space $X$, typically
a Banach space or, more generally, a complete metric space \cite{BV,CV,H,T2}.
Notice that the parameter $t$ could be regarded as discrete, if dealing with a difference equation,
or as continuous, in the case of a differential equation.

When global existence and uniqueness of solutions can be proved, dynamical systems arise as
solution operators assigning to a certain initial condition $x\in X$ the corresponding
solution $x(t)=S(t)x$ of the evolution problem under consideration. Unfortunately, in many
instances, uniqueness of solutions may be hard to prove, or even out of reach. In this case,
one has to deal with a so-called multivalued semigroup of operators, for which
$S(t)x$ is the set of all possible solutions at time $t$. Multivalued semigroups have
been investigated by many authors, and are particularly powerful in the study of abstract differential inclusions,
doubly nonlinear equations, gradient flows and stochastic partial differential equations  \cite{B,BB,CLM,MV,RSS,SSS}.

From the large time behavior viewpoint, the most  relevant object is the so-called \emph{global attractor},
namely, the unique compact subset of the phase space which is at the same time invariant and attracting. As noted in
\cite{W1}, it is of crucial importance to understand whether the longterm dynamics of  a system
possessing global attractor can be properly approximated by discrete attractors of discrete dynamical systems
generated, for example, by numerical schemes associated to the evolution problem under concern. 
This issue has been widely investigated and the interested reader is referred to, e.g., \cite{HLR,HS2000,ju:02,KV,S1989,st,TXW}.

\subsection{The physical model and its approximation} Let $\Omega\subset \R^2$ be a bounded domain with smooth
boundary $\partial \Omega$. For $t\geq0$ and $\nu>0$, we consider  the two-dimensional Navier-Stokes equations \cite{S,T1}
\begin{equation}\label{eq:NS}
\begin{cases}
\pt u-\nu\Delta u + (u\cdot \nabla)u +\nabla p=f,\\
\dive u =0,
\end{cases}
\end{equation}
where $f$ is an autonomous incompressible forcing term. The system is supplemented with the nonslip  boundary condition
\begin{equation}\label{eq:NSBC}
u(x,t)|_{x\in\partial \Omega}=0,
\end{equation}
and the initial condition
\begin{equation}\label{eq:NSIC}
u(x,0)=u_0(x),
\end{equation}
along with its time discretization provided by the fully implicit Euler scheme \cite{T3}
\begin{equation}\label{eq:phyEU}
\frac{u^n-u^{n-1}}{k}-\nu \Delta u^n + (u^n\cdot \nabla)u^n +\nabla p^n=f, \qquad u^0=u_0.
\end{equation}
Due to its nonlinear nature, uniqueness of solution to such a numerical approximation can be proved only
by restricting the time-step parameter $k>0$ to be small enough.  To be more precise, such a restriction
depends on the initial datum $u_0$ and does not therefore allow to define a single-valued discrete semigroup
of operators in the classical sense. Nonetheless, we will be able to show that, in fact, this difficulty can
be overcome by defining a family of \emph{multivalued} discrete semigroups $\{S_k,\, 0<k\leq \kappa_1\}$, where
$\kappa_1>0$ is constant  and independent of $u_0$.

Taking advantage of previous results contained in \cite{TW}, we will address the issue of the existence of a family of discrete global attractors
$\A_k$ of $S_k$ and we will prove that $\A_k\to \A$ as $k\to 0$ in a suitable sense,
where $\A$ is the global attractor of the single-valued dynamical system generated by
\eqref{eq:NS}--\eqref{eq:NSIC}.

\subsection{Structure of the paper} In the next section we develop,
along the lines of \cite{MV}, the abstract machinery of multivalued dynamical systems
needed to prove the results described above. In particular, we will
focus on the existence of multivalued global attractors, and prove a convergence
theorem for families of discrete attractors depending on a parameter.
Section \ref{sec:appl} is dedicated to the applications of the abstract theory
to the Navier-Stokes equations and their approximation. It is shown that the fully
implicit Euler scheme generates a  multivalued discrete dynamical system
whose asymptotic dynamics is properly related to the one of the Navier-Stokes equations.

\section{The Abstract Framework}\label{sec:abs}
\noindent We here develop in an abstract way the main tools needed for the study of
multivalued semigroups and their asymptotic behavior. Some results are more or less already known
thanks to the works \cite{B} and \cite{MV}. We also refer to \cite{CMR} and to the more recent work \cite{GS}
for a fairly complete overview on the subject. Except for the approximation results,
equivalent statements of the theorems in this section can be found in \cite{MV}. 
We report here alternative proofs for completeness and review purposes.

While the above papers mainly dealt with continuous-time semigroup, we are more interested
in discrete dynamical systems generated by numerical
schemes for which uniqueness of solutions may not hold. In particular, we will address   the question
of under what conditions a discrete multivalued semigroup can approximate the longtime behavior
of a continuous one.

\subsection{Multivalued Semigroups}
Let $(X,\|\cdot\|)$ be a real Banach space, and let $\TT$ be either $\R^+=[0,\infty)$ or $\N$. A one-parameter family of set-valued maps $S(t):2^X\to 2^X$
is a \emph{multivalued semigroup} (m-semigroup) if it satisfies the following properties:

\begin{enumerate}[label=(S.\arabic*)]
	\item $S(0)$ is the identity on $2^X$; \label{S.1}
	\item $S(t+\tau)=S(t)S(\tau)$, for all $t,\tau \in \TT$. \label{S.2}
\end{enumerate}

To simplify the notation, if $x\in X$ we will write $S(t)x$ in place of $S(t)\{x\}$ and, as customary, for any set $\B\in 2^X$, we will assume
$$
S(t)\B=\bigcup_{x\in\B}S(t)x.
$$
The m-semigroup is said to be \emph{closed} if it fulfills the further property:
\begin{enumerate}[label=(S.\arabic*),resume]
	\item  $S(t)$ is a closed map for every $t\in\TT$, meaning that if $x_n\to x$ and
	$y_n\in S(t)x_n$ is such that $y_n\to y$, then $y\in S(t)x$. \label{S.3}
\end{enumerate}

\begin{remark}
The notion of continuity of a set-valued map is not as immediate as in the single-valued
case. An m-semigroup is called \emph{upper semicontinuous} if given $x\in X$ and
a neighborhood $\U(S(t)x)$ of $S(t)x$, there exists $\delta>0$ such that
$$
\|x-y\|<\delta\qquad \Rightarrow\qquad S(t)y\subset \U(S(t)x).
$$
On the other hand, $S(t)$ is defined to be \emph{lower semicontinuous} if given
$x_n\to x$ and $y\in S(t)x$, there exists $y_n\in S(t)x_n$ such that $y_n\to y$.
Finally, $S(t)$ is \emph{continuous} if it is at the same time lower and upper semicontinuous.
In general, it is not true that a continuous m-semigroup is closed (as it is in the single-valued case),
the problem being that $S(t)x$ might not be a closed set (if $S(t)x$ is closed for any $x\in X$, the $S(t)$ is said to have closed values).
Nonetheless, if $S(t)$ is upper semicontinuous and has closed values, then $S(t)$ is closed (see \cite{AF}).
\end{remark}
The \emph{positive orbit} of $\B$, starting at $t\in\TT$, is the set
$$
\gamma_t(\B)=\bigcup_{\tau\geq t}S(\tau)\B,
$$
where we agree to set $\gamma(\B)=\gamma_0(\B)$. A function $y:\TT\to X$ is said to be a trajectory starting at $y_0\in X$
if $y(0)=y_0$ and $y(t+\tau)\in S(t)y(\tau)$ for every $t,\tau\in \TT$.

\subsection*{Limit Sets}
For any $\B\in2^X$, the set
$$
\omega(\B)=\bigcap_{t\in\TT}\overline{\gamma_t(\B)}
$$
is called the \emph{$\omega$-limit set} of $\B$. The following characterization of $\omega(\B)$ holds true in the
multivalued case, and the proof is identical to the single-valued counterpart.
\begin{lemma}\label{lemma:omega}
A point $x\in X$ belongs to  $\omega(\B)$ if and only if there are sequences $t_n\to \infty$ and $x_n\in S(t_n)\B$ 
such that $x_n\to x$ as $n\to\infty$.
\end{lemma}

A nonempty set $\B\in 2^X$ is \emph{invariant} for $S(t)$ if
$$
S(t)\B=\B, \qquad \forall t\in\TT.
$$
If $S(t)\B\subset \B$, then $\B$ is said to be positively invariant.

\subsection*{Dissipativity}
A set $\B_0\in 2^X$ is an \emph{absorbing set} for the m-semigroup $S(t)$
if for every bounded set $\B\in 2^X$ there exists $t_\B\in \TT$ such that
$$
S(t)\B\subset \B_0, \qquad \forall t\geq t_\B.
$$
Given two nonempty sets $\B,\C\in 2^X$, the Hausdorff semidistance between $\B$ and $\C$ (in $X$) is defined as
$$
\dist(\B,\C)=\sup_{b\in\B}\inf_{c\in\C}\|b-c\|.
$$
Notice that, in general, $\dist(\B,\C)\neq \dist(\C,\B)$. 
A nonempty set $\C\in 2^X$ is \emph{attracting} if for every bounded set $\B$ we have
$$
\lim_{t\to \infty}\dist(S(t)\B,\C)=0.
$$
The m-semigroup $S(t)$ is called \emph{dissipative} if it possesses a bounded absorbing set.
Some authors prefer to  require compactness (instead of only boundedness) in the notion of
dissipativity, a strategy which has been successful expecially in parabolic problems \cite{R}.

\subsection*{The Global Attractor}
A nonempty compact set $\A\in 2^X$ is said to be the \emph{global attractor} of $S(t)$ if
\begin{enumerate} [label=(A.\arabic*)]
	\item  $\A$ is invariant;  \label{A.1}
	\item  $\A$ is an attracting set. \label{A.2}
\end{enumerate}

\begin{remark}
The global attractor, if it exists, is necessarily unique. Moreover, it enjoys the following maximality
and minimality properties:
\begin{enumerate}
	\item[(i)] let $\tilde\A$ be a bounded set satisfying  \ref{A.1}. Then $\A\supset \tilde\A$;
	\item[(ii)] let $\tilde\A$ be a closed set satisfying \ref{A.2}. Then $\A\subset \tilde\A$.
\end{enumerate}
\end{remark}

In order to state the result on the existence of the global attractor, we need a definition. Given
a bounded set $\B\in 2^X$, the \emph{Kuratowski measure of noncompacteness} $\alpha(\B)$ of
$\B$ is defined as
$$
\alpha(\B)=\inf\big\{\delta\, :\, \B\ \text{has a finite cover by balls of } X \text{ of diameter less than } \delta\big\}.
$$
We list hereafter some properties of $\alpha$.
\begin{enumerate}[label=(K.\arabic*)]
	\item $\alpha(\B)=\alpha(\overline{\B})$; \label{K.1}
	\item $\B_1\subset \B_2$ implies that $\alpha(\B_1)\leq \alpha(\B_2)$; \label{K.2}
	\item $\alpha(\B)=0$ if and only if $\overline{\B}$ is compact; \label{K.3}
	\item if $\{\B_t\}_{t\in\TT}$ is a family of nonempty closed sets such that $\B_{t_2}\subset \B_{t_1}$ for
			$t_2>t_1$ and $\lim_{t\to\infty}\alpha(\B_t)=0$, then $\B=\bigcap_{t\in\TT} \B_t$ is nonempty and compact; \label{K.4}
	\item if $\{\B_t\}_{t\in\TT}$ and $\B$ are as above, given any $t_n\to\infty$ and any $x_n\in \B_{t_n}$, there exist
			$x\in \B$ and a subsequence $x_{n_k}\to x$. \label{K.5}
\end{enumerate}

\begin{theorem}\label{teo:attr}
Suppose that the closed m-semigroup $S(t)$ possesses a bounded absorbing set $\B_0\in 2^X$ and
\begin{equation}\label{eq:asymptcpt}
\lim_{t\to\infty}\alpha(S(t)\B_0)=0.
\end{equation}
Then $\omega(\B_0)$ is the global attractor of $S(t)$.
\end{theorem}

\begin{proof}
First, we prove that $\omega(\B_0)$ is nonempty and compact. Let $t_0\in\TT$ be such that
$$
S(\tau)\B_0\subset \B_0, \qquad \forall \tau\geq t_0.
$$
If $t\geq t_0$, we have
\begin{align*}
\gamma_t(\B_0)&=\bigcup_{\tau\geq t}S(\tau) \B_0=\bigcup_{\tau\in\TT} S(\tau+t)\B_0\\
&=\bigcup_{\tau\in \TT} S(t-t_0)S(\tau+t_0)\B_0\subset \bigcup_{\tau\in \TT} S(t-t_0)\B_0=S(t-t_0)\B_0.
\end{align*}
Therefore, in light of \ref{K.1}, \ref{K.2} and \eqref{eq:asymptcpt},
$$
\lim_{t\to\infty} \alpha(\overline{\gamma_t(\B_0)})=\lim_{t\to\infty} \alpha(\gamma_t(\B_0))=0.
$$
Since the sets $\gamma_t(\B_0)$ are nested, from \ref{K.4}  we conclude that
$$
\omega(\B_0)=\bigcap_{t\in\TT}\overline{\gamma_t(\B_0)}
$$
is nonempty and compact.

To prove $\omega(\B_0)$ is attracting, argue by contradiction, and assume that there exist $\eps>0$ and sequences
$t_n\to\infty$ and $y_n\in S(t_n)\B_0$ such that
$$
\inf_{y\in \omega(\B_0)}\|y_n-y\|\geq \eps.
$$
Since $y_n\in S(t_n)\B_0$, it follows that $y_n\in \gamma_{t_n}(\B_0)$. By properties \ref{K.4} and \ref{K.5}, there exist $z\in\omega(\B_0)$
and a subsequence $y_{n_k}\in S(t_{n_k})\B_0$ such that $y_{n_k}\to z$, which is a contradiction.

It remains to show that $\omega(\B_0)$ is invariant. Let $t_n\to\infty$ and $y_n\in S(t_n)\B_0$ be any sequences. We
claim that there exists $y\in\omega(\B_0)$ such that $y_n\to y$ up to a subsequence. Since $\omega(\B_0)$ is attracting,
we know that
$$
\lim_{n\to\infty}\dist (S(t_n)\B_0,\omega(\B_0))=0.
$$
As a consequence,
$$
\lim_{n\to\infty}\inf_{z\in\omega(\B_0)} \|y_n-z\|=0,
$$
from which we deduce the existence of a sequence $z_n\in\omega(\B_0)$ such that
$$
\lim_{n\to\infty}\|y_n-z_n\|=0.
$$
From the compactness of $\omega(\B_0)$ we readily get $y\in \omega(\B_0)$ and a subsequence $n_k$
such that $z_{n_k}\to y$, and, in turn, $y_{n_k}\to y$.

Let $x\in \omega(\B_0)$ and $t\in\TT$. By Lemma \ref{lemma:omega},
there exist sequences $t_n\to\infty$ and $y_n\in S(t_n)\B_0$ such that $y_n\to x$ as $n\to \infty$.
Since
$$
S(t_n)\B_0=S(t)S(t_n-t)\B_0,
$$
we obtain that $y_n\in S(t)z_n$, with $z_n\in S(t_n-t)\B_0$. By the above claim, there
exists $y\in\omega(\B_0)$ such that $z_n\to y$ up to a subsequence. The fact that
$S(t)$ is closed then yields the following implication
$$
z_n\to y, \qquad S(t)z_n\ni y_n\to x \qquad \Rightarrow \qquad x\in S(t)y,
$$
namely the inclusion $\omega(\B_0)\subset S(t)\omega(\B_0)$. Turning to the opposite
one, we now know that
$$
S(t)\omega(\B_0)\subset S(\tau)S(t)\omega(\B_0)=S(t+\tau)\omega(\B_0)
$$
for every $t,\tau\in \TT$. Since $\omega(\B_0)$ is attracting, for every neighborhood
$\U(\omega(\B_0))$ of $\omega(\B_0)$, there exists $t_\U\in\TT$ such that
$$
S(t+\tau)\omega(\B_0)\subset \U(\omega(\B_0)), \qquad \forall\tau\geq t_\U.
$$
Thus $S(t)\omega(\B_0)\subset \U(\omega(\B_0))$, where $\U(\omega(\B_0))$ is arbitrary, so that
$$
S(t)\omega(\B_0)\subset \overline{\omega(\B_0)}=\omega(\B_0), \qquad \forall t\in\TT,
$$
which concludes the proof.
\end{proof}

\subsection{Discrete Approximation Of Multivalued Semigroups}
We now turn our attention to the approximation of continuous-time m-semi\-groups
by discrete ones, exploring how certain properties that hold true in the discrete
regime are carried over to the continuous limit.

\begin{remark}
Given a set-valued map $S:2^X\to 2^X$, we can define a discrete m-semigroup by
$$
S(n)=S^n, \qquad \forall n\in\N.
$$
Properties  \ref{S.1} and \ref{S.2}  are
trivially satisfied, and, in this case, we will write that $\{S\}_{n\in\N}$
(instead of the redundant $\{S^n\}_{n\in\N}$) is a discrete m-semigroup.
\end{remark}

Turning to our approximation problem, let $\{S(t)\}_{t\in\R^+}$ be a closed m-semigroup, $\kappa_0$ a
positive constant, and consider a family of discrete
closed m-semigroups $\{S_k,\, 0<k\leq \kappa_0\}_{n\in\N}$ where, for each fixed $k$, the map
$S_k:2^X\to 2^X$ satisfies the usual semigroup properties
$$
S_k^0=\Id_{2^X}, \qquad S_k^{n+m}=S_k^{n}S_k^m, \qquad \forall n,m\in\N.
$$
Such maps arise in the study of numerical schemes associated to evolutionary
equations, in which either the discretization problem or the differential system (or both) might
not enjoy any uniqueness property of solutions. From the point of view of the longtime
behavior of solutions, it is therefore interesting to understand
under what conditions the asymptotic features of the continuous-time m-semigroup can be
properly approximated by the discrete ones.

Given two nonempty sets $\B,\C\in 2^X$, we write
$$
\B-\C=\{b-c:\, b\in\B,\, c\in\C\} \qquad \text{and} \qquad \|\B\|=\sup_{b\in\B} \|b\|.
$$
The following theorem is a generalization to m-semigroups of a result proven in \cite{W1}.
\begin{theorem}\label{teo:approx}
Let $S(t)$ be a closed m-semigroup, possessing the global attractor $\A$, and
let $\{S_k,\, 0<k\leq \kappa_0\}_{n\in\N}$ be a family of discrete closed m-semigroups,
with global attractor $\A_k$. Assume the following:
\begin{enumerate} [label=\rm{(H\arabic*)}]
	\item   there exists $\kappa_1\in (0,\kappa_0]$ such that the set
	\begin{equation}
	\K=\bigcup_{k\in (0,\kappa_1]}\A_k
	\end{equation}
	is bounded in $X$;  \label{H1}
	\item there exists $t_0\geq 0$ such that for any $T^\star>t_0$,
	\begin{equation}
	\lim_{k\to 0}\sup_{x\in\A_k,\,nk\in [t_0,T^\star]}\| S_k^nx-S(nk)x\|=0.
	\end{equation}
	\label{H2}
\end{enumerate}
Then,
\begin{equation}
\lim_{k\to 0}\dist (\A_k,\A)=0.
\end{equation}
\end{theorem}

\begin{proof}
Let $\eps>0$ and $k\in(0,\kappa_1]$. Since $\K$ is bounded and $\A$ is attracting, there exists $t_\eps>t_0\geq 0$ such that
$$
\dist(S(t)\K,\A)< \frac{\eps}{2}, \qquad \forall t\geq t_\eps.
$$
Let now $x_k\in\A_k$, and pick
$$
n_k=\Big\lfloor \frac{t_\eps+1}{k} \Big\rfloor.
$$
Thanks to the invariance of $\A_k$ under $S_k$, there exists $y_k\in \A_k$ such that $x_k\in S^{n_k}_k y_k$, and
by \ref{H2},
$$
\|x_k-S(n_k k)y_k\|\leq \| S_k^{n_k}y_k-S(n_kk)y_k\|<\frac{\eps}{2}, \qquad \forall k\leq \kappa_\eps,
$$
with a proper choice of $\kappa_\eps>0$. As a consequence, if $y\in S(n_k k)y_k$,
\begin{align*}
\dist(\A_k,\A)&=\sup_{x_k\in\A_k}\inf_{x\in\A}\|x_k-x\|\leq
\sup_{x_k\in\A_k}\Big[\|x_k-y\|+\dist(y,\A)\Big]\\
&\leq   \sup_{x_k\in\A_k}\Big[\|x_k-S(n_k k)y_k\|+\dist(S(n_k k)y_k,\A)\Big]\\
&\leq \sup_{x_k\in\A_k}\|x_k-S(n_k k)y_k\|+\dist(S(n_k k)\K,\A)<\eps,
\end{align*}
and the theorem is proved.

\end{proof}


\section{Applications to the Navier-Stokes Equations}\label{sec:appl}
\noindent In this section, we apply the above abstract framework to the fully implicit
Euler approximation of the Navier-Stokes equations. It is known that such a scheme
\emph{does not} generate a single-valued discrete semigroup, since uniqueness 
of solutions holds under a restriction on the time-step parameter which depends
on the initial datum.
To circumvent this difficulty, we will show that such a scheme 
generates a family of closed discrete m-semigroups depending on the time-step parameter,
whose related attractors converge to the attractor of the dynamical system generated by the Navier-Stokes equations.

\subsection{Function spaces}
Let $\Omega\subset \R^2$ be a bounded domain with smooth boundary $\partial \Omega$. For $p\in[1,\infty]$ and
$k\in\N$, we denote
by $\mathbf{L}^p(\Omega)=\{L^p(\Omega)\}^2$, $\mathbf{H}^k(\Omega)=\{H^k(\Omega)\}^2$,
and $\mathbf{H}_0^k(\Omega)=\{H_0^k(\Omega)\}^2$
the usual Lebesgue and Sobolev spaces of vector-valued functions on $\Omega$. Setting
$$
\D=\big\{u\in C^\infty_0(\Omega, \R^2): \, \dive u =0 \big\},
$$
we consider the usual Hilbert spaces associated with the Navier-Stokes equations
\begin{align*}
&H=\text{closure of } \D \text{ in } \mathbf{L}^2(\Omega),\\
&V=\text{closure of } \D \text{ in } \mathbf{H}^1(\Omega),
\end{align*}
where we denote by $|\cdot|$, $(\cdot,\cdot)$ and  $\|\cdot\|$, $((\cdot,\cdot))$ the norm
and the scalar product in $H$ and in $V$, respectively. Also, we indicate by $V^*$ the
dual space of $V$, endowed with the usual dual norm $\|\cdot\|_*$, and by $\la\cdot,\cdot \ra$ the duality pairing between $V$ and $V^*$.
Calling
$$
P:\mathbf{L}^2(\Omega)=H\oplus H^\perp\to H
$$
the Leray orthogonal projection, the Stokes operator is defined as
$$
A=-P\Delta, \qquad \dom (A)=\mathbf{H}^2(\Omega)\cap  V.
$$
It is well known that the operator $A$ is self-adjoint and strictly positive. Moreover, $\dom (A^{1/2})=V$ and
$$
\|u\|=|\nabla u|=|A^{1/2}u|, \qquad \forall u\in V.
$$
Setting
$$
B(u,v)=P\big[(u\cdot \nabla)v\big],
$$
system \eqref{eq:NS}--\eqref{eq:NSIC} can be rewritten as an abstract evolution equation of the form
\begin{equation}\label{eq:NSABS}
\begin{cases}
\dot{u} +\nu Au+B(u,u)=f,\\
u(0)=u_0,
\end{cases}
\end{equation}
where $f=Pf$, since we are assuming incompressible forcing. As proved in \cite{T2}, problem \eqref{eq:NSABS} generates a continuous and dissipative single-valued dynamical
system $S(t):H\to H$, which possesses the global attractor $\A$, bounded in $V$.

\subsection{The implicit Euler scheme}
A possible time discretization of \eqref{eq:NSABS}
is given by  the fully implicit Euler scheme
\begin{equation}\label{eq:EU}
\frac{u^n-u^{n-1}}{k}+\nu Au^n+B(u^n,u^n)=f, \qquad u^0=u_0,
\end{equation}
where $n\geq 1$ and $k>0$ is the time-step. Our goal in this section is to prove that
\eqref{eq:EU} generates a closed discrete m-semigroup $\{S_k\}_{n\in\N}$. Indeed,
as in the stationary Navier-Stokes problem, a solution to \eqref{eq:EU} is not
in general unique. In particular, uniqueness of solutions may be proved requiring
$k$ to be bounded from above by a constant depending on the 
initial datum. More specifically, such a bound vanishes as the norm of $u_0$ tends
to infinity, and this makes it impossible to properly define a single-valued semigroup acting
on the \emph{whole} phase space.   Nonetheless, by means of the tools devised in Section \ref{sec:abs},
we will still be able to give a description of the longtime behavior of
the discretized Navier-Stokes equations, and discuss its convergence to the
time-continuous asymptotic dynamics.

\subsection*{Notation}
Throughout the section, $C$ and $\Q(\cdot)$  will denote  a \emph{generic} positive constant and
a \emph{generic} increasing positive function, respectively, whose value may change even in the
same line of a certain equation. Unless otherwise
stated, these quantities will be \emph{independent} of $k,n$ and of the initial datum $u_0$. In general, they
might depend on the structural quantities of the system ($\nu, f, \Omega$).

\medskip

Fix $k>0$. For $w\in H$, we look at the problem
\begin{equation}\label{eq:aux}
u+k\nu Au +kB(u,u)=w+kf,
\end{equation}
for which we seek solutions in the following weak sense.
\begin{definition}
A vector $u\in V$ is a solution to \eqref{eq:aux} if
\begin{equation}\label{eq:weaksol}
(u,v)+k\nu ((u,v))+kb(u,u,v)=(w,v)+k(f,v), \qquad \forall v\in V,
\end{equation}
where $b(u,v,w)=\la B(u,v),w\ra$ is the usual trilinear form associated to the weak formulation
of the Navier-Stokes equations.
\end{definition}

\begin{remark}
Recall that the trilinear form $b$ satisfies the following properties:
\begin{align}
|&b(u,v,w)|\leq C|u|^{1/2}|A u|^{1/2}\|v\||w|, \qquad \forall u\in \dom(A), v\in V,w\in H,\label{eq:b1}\\
|&b(u,v,w)|\leq C|u|^{1/2} \|u\|^{1/2} \|v\| |w|^{1/2} \|w\|^{1/2}, \qquad \forall u, v,w\in V,\label{eq:b2}\\
&b(u,v,v)=0, \qquad \forall u,v\in V. \label{eq:b3}
\end{align}
\end{remark}
It is a classical result that a (possibly not unique) solution to \eqref{eq:aux} exists. Moreover,
any solution $u\in V$ satisfies the energy estimate
\begin{equation}\label{eq:en}
|u|^2+k\nu\|u\|^2\leq |w|^2+Ck|f|^2,
\end{equation}
where $C>0$ does not depend on $k$. For every $w\in H $, define the multivalued map $S_k:2^H\to 2^H$ by
$$
S_k w=\{u \in V:\, u \text{ solves \eqref{eq:aux} with time-step } k\}.
$$
Notice that, in light of \eqref{eq:en}, the set $S_k w$ is bounded in $V$ and therefore relatively compact in $H$, thanks to the
compactness of the embedding $V\hookrightarrow H$.

\subsection{The discrete m-semigroup}
Let us consider the discrete m-semigroup $\{S_k\}_{n\in\N}$ generated by $S_k$. It is now clear that,
for every $n\in \N$,
$$
S_k u^{n-1}=\{u^n\in H:\, u^n \text{ solves \eqref{eq:EU} with time-step } k\}.
$$
Also, a vector $u^n\in S_k^n u_0$ if and only if there exists a sequence of elements
$(u^0,u^1, \ldots, u^{n-1},u^n)$ such that $u^i\in S_k u^{i-1}$ for every $i=1,\ldots, n$ and $u^0=u_0$.
From \eqref{eq:en}, we infer that any $u^i\in S_k u^{i-1}$ satisfies the energy estimate
$$
|u^i|^2+k\nu\|u^i\|^2\leq |u^{i-1}|^2+Ck|f|^2,
$$
and, inductively, any $u^n\in S_k^n u^0$ fulfills the bound
\begin{equation}\label{eq:ind}
|u^n|^2+k\nu\sum_{i=1}^n \|u^i\|^2\leq |u^0|^2+Ckn|f|^2.
\end{equation}
We then have the following theorem.
\begin{theorem}
The multivalued map $S_k$ associated to the implicit Euler scheme \eqref{eq:EU} generates
a closed discrete m-semigroup $\{S_k\}_{n\in\N}$.
\end{theorem}
\begin{proof}
Since properties \ref{S.1}--\ref{S.2} are satisfied by definition, all we need to prove is
that $S_k^n$ is a closed multivalued map for each $n\in \N$.
As $j\to \infty$, let $u^0_j\to u^0$ in $H$ and $u^n_j\in S^n_k u^0_j$ with
$u^n_j\to u^n$ in $H$. We have to show that $u^n\in S_k^nu^0$.

Since $u^n_j\in S^n_k u^0_j$, there exists a sequence
$(u_j^0, u_j^1,\ldots, u_j^{n-1},u_j^n)$ where $u_j^i\in S_k u_j^{i-1}$ is a solution to
\begin{equation}\label{eq:weak}
(u_j^i,v)+k\nu ((u_j^i,v))+kb(u_j^i,u_j^i,v)=k(f,v)+(u_j^{i-1},v), \qquad \forall v\in V.
\end{equation}
Also, the fact that $u^0_j\to u^0$ in $H$ implies the existence of a positive number $M$ such that
$$
\sup_j |u_j^0|^2\leq M.
$$
In view of \eqref{eq:ind}, we obtain the bound
$$
|u_j^i|^2+k \nu \sum_{\ell=1}^i  \|u_j^\ell\|^2\leq Cki|f|^2 +M.
$$
Thus, for every $i=1,\ldots, n$, we have the following convergences (up to not relabeled subsequences) as
$j\to\infty$:
$$
u^i_j\to u^i, \text{ strongly in } H \text{ and weakly in } V.
$$
Now, passing to the limit in \eqref{eq:weak}, we readily get that
$$
(u^i,v)+k\nu ((u^i,v))+kb(u^i,u^i,v)=k(f,v)+(u^{i-1},v),\qquad \forall v\in V.
$$
As a consequence, $u^i\in S_k u^{i-1}$ for each $i=1,\ldots, n$.  But then, $u^n\in S_k u^{n-1}\subset S_k^n u^0$, so
$S_k^n$ is a closed map for every $n\in \N$.
\end{proof}

From the energy estimate \eqref{eq:ind} and the fact that
a closed map has necessarily closed values, we have the
following straightforward consequence.

\begin{corollary}
The discrete m-semigroup $\{S_k\}_{n\in\N}$ has compact values, namely,
the set $S_k^n u_0$ is compact in $H$ for every $u_0\in H$.
\end{corollary}

\subsection{Earlier Contributions}
To continue our study, we first collect some results obtained in \cite{TW}.  The first two concern initial data $u_0\in H$.
There exists $\kappa_0>0$, independent of $u_0,n,k$, such that the following hold:

\begin{enumerate}[label=(D.\arabic*)]
	\item For every $k>0$,
	$$
	|S_k^n u_0|^2\leq (1+k)^{-n}|u_0|^2+C|f|^2, \qquad \forall n\geq 0.
	$$\label{D.1}
	\item Let $k\in(0,\kappa_0]$. There exists a constant $R_0>0$ with the following property: for every
	$R\geq 0$, there exists $t_0=t_0(R)\geq 0$ such that
	$$
	|S_k^n u_0|\leq R_0, \qquad \forall nk\geq t_0,
	$$
	whenever $|u_0|\leq R$. Both $R_0$ and $t_0$ can be explicitly computed and do not depend on $n$ and $k$.
	In other words, the set
	$$
	\B_0=\{v\in H:\, |v|\leq R_0\}
	$$
	is a bounded absorbing set for $\{S_k\}_{n\in\N}$.\label{D.2}
\end{enumerate}
We now turn our attention to initial data $u_0\in V$ and recall the main result derived in \cite{TW}, tailored
to our case.
\begin{theorem}\label{thm:TW}
Suppose $\|u_0\|\leq R$, and let the time-step $k$ be such that
\begin{equation}\label{eq:time}
	k\leq \kappa_\star(R)=\min\Big\{\kappa_0, \frac{1}{\Q(R)},C \Big\}.
\end{equation}
Then the estimate
\begin{equation}\label{D.3}
\|S_k^n u_0\|\leq \Q(R)
\end{equation}
holds true for every $n\geq 1$.
\end{theorem}

\begin{remark}
Since the issue of non-uniqueness of solutions to \eqref{eq:aux} is the main motivation for this work,
let us briefly discuss one way to recover uniqueness of solution to \eqref{eq:EU}, in the sense of
\eqref{eq:weaksol} with $u=u^n$ and $w=u^{n-1}$.
Let $u^n_1$ and $u^n_2$ be two solutions corresponding
to the same initial data $u_0\in V$, let $R\geq 0$  be such that $\|u_0\|\leq R$,
and set $u^n=u^n_1-u^n_2$. If $k\leq \kappa_\star(R)$, from \eqref{eq:weaksol} and \eqref{eq:b3} we learn that
$$
|u^n|^2+k\nu\|u^n\|^2=-kb(u^n,u^n_2,u^n).
$$
Now, using \eqref{eq:b2} and Theorem \ref{thm:TW}, for any $n\geq 1$ we obtain 
\begin{align*}
|u^n|^2+k\nu\|u^n\|^2&\leq Ck |u^n|\|u^n\|\|u_2^n\|\leq \Q(R)k|u^n|\|u^n\|\\
&\leq \frac12 |u^n|^2+\Q(R)^2k^2\|u^n\|^2.
\end{align*} 
Therefore, if we require
$$
k\leq \min\Big\{\kappa_\star(R),\frac{\nu}{2\Q(R)^2}\Big\},
$$
we can conclude that
$$
|u^n|^2+k\nu\|u^n\|^2\leq 0.
$$
This estimate clearly yields uniqueness of solutions. However, two main drawbacks arise. Firstly, we necessarily need $u_0\in V$,
which rules out the possibility of defining a single-valued semigroup on the natural phase space $H$ of weak solutions.
Secondly, the restriction on $k$ depends on $\|u_0\|$, and thus uniqueness of solutions depends, in the end, on the single trajectory
chosen and not uniformly with respect to the initial datum. This is why we decided to tackle this problem exploiting the machinery
of multivalued semigroups.
\end{remark}

\subsection{The discrete global attractors}
The higher order estimate in \eqref{D.3} is not enough to conclude the existence of the global
attractor for the discrete m-semigroup $S_k$, since it requires the initial data $u_0$ to be in the
more regular space $V$. Moreover, \eqref{eq:time} shows a dependence of $\kappa_\star$ on
the initial data, which turns out to be unsatisfactory in the approximation process devised in
Theorem \ref{teo:approx}.  We now show how to overcome this difficulty. First of all, thanks to the existence of a bounded absorbing set, it is
natural to consider only initial data $u_0\in \B_0$.

\begin{lemma}\label{lem:elle}
Let $k\in (0,\kappa_0]$, and consider a sequence $(u^0,u^1, \ldots)$, where $u^0=u_0\in\B_0$ and
$u^i\in S^i_k u^0$. Then, there exists $\ell_k\in \N$ such that
\begin{equation}
\|u^{\ell_k}\|\leq R_\star,
\end{equation}
where $R_\star>0$ does not depend on $n,k$ and $u_0$.
\end{lemma}

\begin{proof}
Let
$$
n_k=\Big\lfloor \frac{1}{k} \Big\rfloor+1.
$$
Estimate \eqref{eq:ind} immediately implies
$$
k\nu\sum_{i=1}^{n_k} \|u^i\|^2\leq R_0^2+Ckn_k|f|^2.
$$
Arguing by contradiction, we infer that there exists $\ell_k\in \{1,\ldots, n_k\}$ such that
$$
kn_k\nu \|u^{\ell_k}\|^2\leq R_0^2+Ckn_k|f|^2.
$$
Hence,
$$
\|u^{\ell_k}\|^2\leq \frac{R_0^2}{kn_k\nu}+C|f|^2.
$$
Since $kn_k\geq 1$, the proof ends by setting
$$
R^2_\star=\frac{R_0^2}{\nu}+C|f|^2.
$$
\end{proof}
Having in mind condition \eqref{eq:time}, we now fix $\kappa_1=\kappa_\star(R_\star)\leq \kappa_0$.
Combining together \eqref{D.3} and the above Lemma \ref{lem:elle} we obtain the following.
\begin{corollary}
Let $k\in (0,\kappa_1]$, and consider a sequence $(u^0,u^1, \ldots)$, where $u^0=u_0\in\B_0$ and
$u^i\in S^i_k u^0$. Then, there exists $\ell_k\in \N$ such that
\begin{equation}
\|u^{\ell_k+n}\|\leq R_1, \qquad \forall n\geq1.
\end{equation}
As a consequence, for every $k\in (0,\kappa_1]$, there exists $n_k\in \N$ such that
\begin{equation}\label{eq:R1}
\|S_k^{n_k+n}u_0\|\leq R_1, \qquad \forall n\geq1.
\end{equation}
\end{corollary}
\begin{proof}
The constant $R_1=\Q(R_\star)$ is given by \eqref{D.3}, and it is clearly independent of $n,k$
and $u_0$.
In the first estimate in the statement of the corollary, $\ell_k$ might depend on the particular sequence $(u^0,u^1,\ldots)$ which
originates from $u^0=u_0$. As in the proof of Lemma \ref{lem:elle}, the choice
$$
n_k=\Big\lfloor \frac{1}{k} \Big\rfloor+1,
$$
together with \eqref{D.3} and the fact that $\ell_k\leq n_k$, takes away this dependence and allows to obtain the uniform estimate \eqref{eq:R1}.
\end{proof}
Notice that \eqref{eq:R1} can be rewritten in the equivalent way
$$
\|S_k^{n}u_0\|\leq R_1, \qquad \forall n\geq n_k+1,
$$
and from the definition of $n_k$, we have that, in particular,
\begin{equation}\label{eq:entering}
\|S_k^{n}u_0\|\leq R_1, \qquad \forall nk\geq 1+2\kappa_1.
\end{equation}
We summarize the above observations in the next theorem, which, in fact, improves the results in \cite{TW}.
\begin{theorem}\label{prop:Vabs}
Let $\kappa_1>0$ as above and  $k\in(0,\kappa_1]$.
There exists a constant $R_1>0$ with the following property: for every
$R\geq 0$, there exists $t_1=t_1(R)\geq 0$ such that
\begin{equation}\label{eq:BBBB1}
\|S_k^n u_0\|\leq R_1, \qquad \forall nk\geq t_1,
\end{equation}
whenever $|u_0|\leq R$. Both $R_1$ and $t_1$ can be explicitly computed and do not depend on $n$ and $k$.
Hence, the set
\begin{equation}
\B_1=\{v\in V:\, \|v\|\leq R_1\}
\end{equation}
is a $V$-bounded absorbing set for $\{S_k\}_{n\in\N}$.
\end{theorem}

\begin{proof}
Fix $k\in(0,\kappa_1]$ and let $|u_0|\leq R$. From \ref{D.2}, there exists $t_0=t_0(R)$ such that
$$
S_k^n u_0 \in \B_0, \qquad \forall nk\geq t_0.
$$
In view of \eqref{eq:entering}, setting $t_1=t_1(R)=t_0(R)+ 1+2\kappa_1$, we obtain
$$
S_k^n u_0 \in \B_1, \qquad \forall nk\geq t_1,
$$
concluding the proof.
\end{proof}

\begin{remark}
The result contained in the above Theorem \ref{prop:Vabs} improves the one in \cite{TW} in two directions.
On one hand, we only require the initial data to be in $H$. This shows a regularization property
analogous to the one enjoyed by the solution to the Navier-Stokes equations. On the other hand, we obtain
a uniform restriction on the time-step $k$, independent of the initial data.
\end{remark}

Thanks to the above results, the discrete m-semigroup $\{S_k\}_{n\in \N}$
satisfies the assumptions of Theorem \ref{teo:attr}.
\begin{proposition}
For every $k\in (0,\kappa_1]$, there exists the global attractor $\A_k$ of the m-semigroup $\{S_k\}_{n\in \N}$.
\end{proposition}

\begin{remark}
The global attractor $\A_k$ being the smallest closed attracting set of the phase space, from \eqref{eq:BBBB1} we
obtain the inclusion
$$
\A_k\subset \B_1.
$$
The fact that $R_1$ does not depend on $k$ then yields
\begin{equation}\label{eq:unifbdd}
\bigcup_{k\in(0,\kappa_1]}\A_k \subset \B_1.
\end{equation}
Hence, the attractors $\A_k$ enjoy a uniform regularity property.
\end{remark}

\subsection{The attractor approximation}\label{appr}
In this paragraph, we prove that the longterm behavior of the semigroup $S(t)$ generated by the
Navier-Stokes equations \eqref{eq:NSABS}
is approximated, in the sense of Theorem \ref{teo:approx}, by that of the discrete m-semigroup
related to the fully implicit Euler scheme \eqref{eq:EU}. The main result of this section reads as follows.
\begin{theorem}\label{teo:apprrrr}
The family of attractors $\{\A_k\}_{k\in(0,\kappa_1]}$  converges, as $k\to 0$, to  $\A$, namely,
\begin{equation}
\lim_{k\to 0}\dist(\A_k,\A)=0,
\end{equation}
where $\dist$ denotes the Hausdorff semidistance in $H$.
\end{theorem}

Our goal is to apply Theorem \ref{teo:approx} to obtain the convergence of the discrete attractors
$\A_k$ to the continuous time attractor $\A$ of the semigroup $S(t)$.
By virtue of \eqref{eq:unifbdd}, assumption \ref{H1} is automatically
satisfied. The remaining of the section  is devoted to the verification of the uniform convergence required
by \ref{H2}.
Since \ref{H2} involves an estimate in terms of initial data belonging to $\A_k$, till the end of the section
we will assume
$$
u_0\in \B_1.
$$
Define $\kappa_2=\min\{\kappa_1,\kappa_\star(R_1)\}$, where $\kappa_\star$ is given by \eqref{eq:time}. 
Notice that since $\B_1$ is absorbing, in view of Theorem \ref{thm:TW}  we have the following uniform estimate
\begin{equation}\label{eq:unifen}
\sup_{k\in (0,\kappa_2]}\sup_{n\geq 0}\|S^n_k u_0\|\leq C.
\end{equation}
which, in turn, implies
\begin{equation}\label{eq:diffu}
\sum_{n=i}^m\|u^n-u^{n-1}\|^2\leq Ck (m-i+1)+C, \qquad \forall i=1,\ldots , m,
\end{equation}
where  $u^n\in S_k u^{n-1}$ for every $n=i,\ldots, m$. Indeed, multiplying
\eqref{eq:EU} by $2kA u^n$, we obtain
\begin{align*}
\|u^n\|^2-\|u^{n-1}\|^2+\|u^n-u^{n-1}\|^2&+2\nu k|Au^n|^2 \\
&+2kb(u^n,u^n,Au^n)=2k(f,Au^n).
\end{align*}
Estimating the trilinear form using \eqref{eq:b1} and the bound \eqref{eq:unifen}, we infer that
$$
2kb(u^n,u^n,Au^n)\leq  2Ck|u^n|^{1/2}\|u^n\||Au^n|^{3/2}\leq \frac{\nu k}{2}|Au^n|^2+Ck.
$$
Also,
$$
2k(f,Au^n)\leq 2k|f||Au^n|\leq \frac{\nu k}{2}|Au^n|^2+ Ck.
$$
Thus,
$$
\|u^n\|^2-\|u^{n-1}\|^2+\|u^n-u^{n-1}\|^2+\nu k|Au^n|^2\leq Ck.
$$
Now, summing over $n=i,\ldots, m$ and neglecting the positive term $|Au^n|^2$, we get
$$
\|u^m\|^2-\|u^{i-1}\|^2+\sum_{n=i}^m\|u^n-u^{n-1}\|^2\leq Ck(m-i+1).
$$
Hence, a further application of \eqref{eq:unifen}  entails \eqref{eq:diffu}.

For any $k>0$, we define the piecewise constant and the piecewise linear functions
$$
u_k(t)=u^n, \qquad t\in [(n-1)k,nk)
$$
and
$$
\tilde{u}_k(t)=u^n+\frac{t-nk}{k}(u^n-u^{n-1}), \qquad t\in[(n-1)k,nk).
$$
Notice that $\tilde{u}_k(nk)=u^n$. Also, it is easily seen that $\tilde{u}_k$ solves
\begin{equation}\label{eq:tilde}
\dot{\tilde{u}}_k+\nu A\tilde{u}_k+B(\tilde{u}_k,\tilde{u}_k)=f+\Psi_k,
\end{equation}
where
$$
\Psi_k(t)=\nu A(\tilde{u}_k(t)-u_k(t))+B(\tilde{u}_k(t),\tilde{u}_k(t))-B(u_k(t),u_k(t)).
$$

\begin{lemma}\label{lem:psi}
For any $T^\star>0$, $\Psi_k \in L^2(0,T^\star; V^*)$ and
\begin{equation}
\|\Psi_k\|^2_{L^2(0,T^\star; V^*)}\leq k\Q(T^\star).
\end{equation}
\end{lemma}
\begin{proof}
Let $v\in V$ be such that $\|v\|\leq 1$, and let $t\in [(n-1)k,nk)$ be fixed. In light of \eqref{eq:b2}, we have
\begin{align*}
|\la B(\tilde{u}_k,\tilde{u}_k)-B(u_k,u_k),v\ra|&=|b(\tilde{u}_k,\tilde{u}_k-u_k,v)+b(\tilde{u}_k-u_k,u_k,v)|\\
&\leq C(\|\tilde{u}_k\|+\|u_k\|)\|\tilde{u}_k-u_k\|.
\end{align*}
Since the uniform bound \eqref{eq:unifen} implies
$$
\|u_k(t)\|=\|u^n\|\leq C
$$
and
$$
\|\tilde{u}_k(t)\|\leq\|u^n\|+\Big|\frac{t-nk}{k}\Big| (\|u^n\|+\|u^{n-1}\|)\leq C,
$$
we infer that
$$
|\la B(\tilde{u}_k,\tilde{u}_k)-B(u_k,u_k),v\ra|\leq C\|\tilde{u}_k-u_k\|\leq C\|u^n-u^{n-1}\|.
$$
Clearly,
$$
|\la A(\tilde{u}_k-u_k),v\ra|\leq \|\tilde{u}_k-u_k\|\leq C\|u^n-u^{n-1}\|.
$$
Hence, for $t\in [(n-1)k,nk)$, we can conclude that
$$
\|\Psi_k(t)\|_*\leq C\|u^n-u^{n-1}\|.
$$
Thus, setting $N^\star=\lfloor T^\star/k \rfloor$, by the above bound and \eqref{eq:diffu} we finally get
\begin{align*}
\|\Psi_k\|^2_{L^2(0,T^\star; V^*)}&= \int_0^{T^\star}\|\Psi_k(t)\|_*^2\dt t \leq \sum_{n=1}^{N^\star+1}\int_{(n-1)k}^{nk}\|\Psi_k(t)\|_*^2\dt t\\
&\leq Ck\sum_{n=1}^{N^\star +1}\|u^n-u^{n-1}\|^2\leq k \Q(T^\star).
\end{align*}
Thus, the lemma is proved.
\end{proof}

We are now ready to verify (a slightly stronger version of) assumption \ref{H2} for our discrete m-semigroup, which will
conclude the proof of Theorem \ref{teo:apprrrr}.

\begin{lemma}
For any $T^\star> 0$,
\begin{equation}
\lim_{k\to 0}\sup_{u_0\in \B_1,\,nk\in [0,T^\star]}| S_k^n u_0-S(nk)u_0|=0.
\end{equation}
\end{lemma}

\begin{proof}
Let $u=u(t)=S(t)u_0$ be the solution to \eqref{eq:NSABS}. As shown in \cite{T2}, $S(t)$ satisfies the
uniform energy estimate
\begin{equation}\label{eq:uniS}
\sup_{t\geq 0}\sup_{u_0\in \B_1} \|S(t)u_0\|\leq C.
\end{equation}
For $\tilde{u}_k$ defined as above and $k\in (0,\kappa_2]$, consider the difference $v_k=u-\tilde{u}_k$, which is a solution to
$$
\dot{v}_k +\nu Av_k+B(v_k,u)+B(\tilde{u}_k,v_k)=-\Psi_k.
$$
Testing the above equation by $v_k$, one obtains
$$
\frac12\frac{\dt}{\dt t}|v_k|^2+\nu \|v_k\|^2+b(v_k,u,v_k)=-\la\Psi_k,v_k\ra.
$$
By \eqref{eq:b2} and \eqref{eq:uniS}, the trilinear form can be estimated as
$$
|b(v_k,u,v_k)|\leq C|v_k|\|v_k\|\|u\|\leq\frac{\nu}{4}\|v_k\|^2+C|v_k|^2,
$$
and from the obvious bound
$$
-\la\Psi_k,v_k\ra\leq  \frac{\nu}{4}\|v_k\|^2+C\|\Psi_k\|_*^2,
$$
we derive the differential inequality
$$
\frac{\dt}{\dt t}|v_k|^2\leq C|v_k|^2+C\|\Psi_k\|_*^2.
$$
Since $v_k(0)=0$, an application of the Gronwall inequality together with Lemma \ref{lem:psi} gives
$$
|v_k(t)|^2\leq C\e^{C t} \int_0^t\|\Psi_k(s)\|_*^2\dt s \leq C\e^{C T^\star}\| \Psi_k\|^2_{L^2(0,T^\star;V^*)}\leq k\Q(T^\star).
$$
As a consequence,
$$
|u(t)-\tilde{u}_k(t)|^2\leq k\Q(T^\star).
$$
But then
\begin{align*}
&\lim_{k\to 0}\sup_{u_0\in \B_1,\,nk\in [0,T^\star]}| S_k^n u_0-S(nk)u_0|\\
&\qquad=\lim_{k\to 0}\sup_{u_0\in \B_1,\,nk\in [0,T^\star]}\sup_{u^n\in S_k^n u_0}| u^n-u(n k)|\\
&\qquad=\lim_{k\to 0}\sup_{u_0\in \B_1,\,nk\in [0,T^\star]}\sup_{u^n\in S_k^n u_0}| \tilde{u}_k(nk)-u(nk)|=0,
\end{align*}
and the proof is over.
\end{proof}

\noindent {\textbf{Acknowledgments.}}  This work was partially supported
by the National Science Foundation under the grant NSF-DMS-0906440,
and by the Research Fund of Indiana University.

\end{document}